\documentclass[a4paper,11pt]{amsart}
\usepackage{amsmath,amssymb,amsthm,amscd,xypic}

\title[Cogenerators and Brown representability]{Constructing cogenerators in triangulated categories and Brown representability}

\author{George Ciprian Modoi}
\address{Babe\c s--Bolyai University, Faculty of Mathematics and Computer Science \\
1, Mihail Kog\u alniceanu, 400084 Cluj--Napoca, Romania}
\email{cmodoi@math.ubbcluj.ro}

\thanks{This work was supported by a grant CNCS-UEFISCDI, project number
PN-II-ID-PCE-2012-4-0100.}


\renewcommand{\iff}{if and only if }

\newcommand{\la}{\longrightarrow}

\newcommand{\N}{\mathbb{N}}
\newcommand{\Z}{\mathbb{Z}}
\newcommand{\Q}{\mathbb{Q}}

\DeclareMathOperator{\Hom}{Hom} 
 
\DeclareMathOperator{\Ker}{Ker}
\DeclareMathOperator{\Img}{Im}

\DeclareMathOperator{\holim}{\underleftarrow{\textrm{holim}}}

\newcommand{\A}{\mathcal{A}}

\newcommand{\C}{\mathcal{C}}

\newcommand{\CS}{\mathcal{S}}
\newcommand{\T}{\mathcal{T}}

\newcommand{\U}{\mathcal{U}}

\newcommand{\I}{\mathcal{I}}
\newcommand{\J}{\mathcal{J}}
\newcommand{\G}{\mathcal{G}}
\newcommand{\CL}{\mathcal{L}}

\newcommand{\ModR}{\hbox{\rm Mod-}\!R}
\newcommand{\ModA}{\hbox{\rm Mod-}\!A}

\newcommand{\FlatR}{\hbox{\rm Flat-}R}

\newcommand{\ProjR}{\hbox{\rm Proj-}\!R}

\newcommand{\Ab}{\mathcal{A}b}
\newcommand{\Proj}{\hbox{\rm Proj-}\!}
\newcommand{\Pproj}{\hbox{\rm Pproj-}\!}

\newcommand{\opp}{^\textit{o}}
\newcommand{\Mod}[1]{\hbox{\rm Mod-}\!({#1})}
\newcommand{\md}[1]{\hbox{\rm mod-}\!#1}

\newcommand{\Add}[1]{\mathrm{Add}\left({#1}\right)}
\newcommand{\Prod}[1]{\mathrm{Prod}({#1})}
\newcommand{\Prd}{\mathrm{Prod}}

\newcommand{\Htp}[1]{\mathbf{K}({#1})}

\newcommand{\Cpx}[1]{\mathbf{C}({#1})}
\newcommand{\Loc}[1]{\mathrm{Loc}({#1})}
\newcommand{\Coloc}[1]{\mathrm{Coloc}({#1})}

\theoremstyle{plain}
\newtheorem{thm}{Theorem}[section]
\newtheorem{lem}[thm]{Lemma}
\newtheorem{prop}[thm]{Proposition}
\newtheorem{cor}[thm]{Corollary}

\theoremstyle{definition}

\theoremstyle{remark}

\newtheorem{expl}[thm]{Example}

\begin{document}

\begin{abstract}
For a triangulated category with products we prove a formal criterion in order to satisfy Brown representability for covariant functors. 
We apply this criterion for showing that both homotopy category of projective modules and homotopy category of projective complexes satisfy 
this kind of representability.
\end{abstract}
\maketitle 


\section*{Introduction}

Brown representability theorem is an old and venerable subject. Perhaps its importance comes 
from the fact that it is a tool for producing adjoints for triangulated functors. 
To be more specific, consider a triangulated category $\T$.  
The definition and basic properties of triangulated categories are to be found in the standard 
reference \cite{N}. A (co)homological functor on $\T$ is a 
(contravariant) functor $F:\T\to\A$ into an abelian category which sends triangles to 
long exact sequences. If $\T$ has (co)products, then Brown representability 
for covariant (contravariant) functors means that all (co)homological product--preserving 
(which sends coproducts into products) functors $\T\to\Ab$ are naturally isomorphic to 
$\T(X,-)$ (respectively to $\T(-,X)$) for some $X\in\T$. Clearly Brown representability for covariant and 
contravariant functors are dual to each other. First result of this type was proved by Brown in \cite{B} for 
contravariant functors defined on the homotopy category of spectra. 
From this historical reason Brown 
representability for contravariant functors is also called (direct) 
Brown representability and that 
for covariant functors is called dual; for brevity we shall also say $\T$, respectively $\T\opp$, satisfies 
Brown representability. It is straightforward to check that if $\T$ ($\T\opp$) satisfies Brown representability, 
then every triangulated functor $\T\to\T'$ which preserves coproducts (products) has a right (left) 
adjoint. 

Remark that Brown representability for contravariant functors is better understood than its dual. 
We know when usual triangulated categories satisfy this property. For example, the derived category of a 
Grothendieck abelian category (including the derived category of modules 
or of quasicoherent sheaves) is $\alpha$-compactly generated for some regular cardinal $\alpha$, 
in the sense of \cite[Definition 8.1.6]{N} (note that Neeman calls such categories
well generated). Consequently it satisfies Brown representability for 
contravariant functors, by \cite[Theorem 8.4.2]{N}. In contrast, the homotopy category of an additive category does not 
satisfy Brown representability, unless the initial additive category 
is pure--semisimple (see \cite{MS}). 

Even if Brown representability for covariant functors holds in compactly generated triangulated 
categories, by a result proved by Neeman in \cite[Theorem 8.6.1]{N}, one of the most challenging problem left open in Neeman's book 
is if an $\alpha$-compactly generated triangulated category, where $\alpha$ is a regular cardinal 
greater than $\aleph_0$,
satisfies Brown representability for covariant functors. This paper belongs to a suite of works 
concerned with this subject. In \cite{MP} is given a criterion for Brown representability 
for contravariant functors
which is dualized in \cite{MD}, where it is shown that if $\T\opp$ is deconstractible 
then $\T\opp$ satisfies Brown representability 
(see Proposition \ref{decon} here). The paper \cite{MBD} applies
this result for derived categories of Grothendieck categories which are AB4-$n$, in 
the sense that the $n+1$-th derived functor of the direct product functor is exact. In the paper 
\cite{MKP} it is observed that Neeman's arguments in \cite{NKP} used to construct cogenerators in the 
homotopy category of projective modules may be modified and adapted in order to show that the dual
of this category is deconstructible.  
Here we formalize this last approach. Fortunately the formal results which we obtained 
(Theorem \ref{main-thm} and Corollaries \ref{t-is-u} and \ref{gp-eq-u}) may be applied not only 
to recover the main result in \cite{MKP} (see Theorem \ref{kproj-mod}), but also 
to generalize a formal criterion by Krause for Brown representability for covariant functors 
(Corollary \ref{krauseB}). As another consequence we show that the 
homotopy category of projective representations of a quiver satisfies 
Brown representability for covariant functors 
(Theorem \ref{kproj-ri}). In particular homotopy category of projective complexes of modules 
considered in \cite{S} has to satisfy this property (Example
\ref{kproj-cpx}).  
Finally we note that there is no known example 
of triangulated categories which are not compactly generated but satisfy the hypothesis of 
Krause's criterion. As we have noticed, for the fact that compactly generated triangulated categories satisfy 
Brown representability for covariant functors there are other proofs 
available. In contrast, our criterion applies for triangulated categories which are 
decidedly not compactly generated. Indeed, by \cite[Facts 2.8]{NKP}, the homotopy category of 
projective right modules over a ring which is not left coherent has to be 
only $\aleph_1$-compactly generated.

\section{A criterion for Brown representability for covariant functors}\label{general-results}

Let $\T$ be a triangulated category and denote by $\Sigma$
 its suspension functor. Let $\CS\subseteq\T$ be a set of objects.
 We 
put \[\CS^\perp=\{X\in\T\mid\T(S,X)=0\hbox{ for all }S\in\CS\},\] 
\[^\perp{\CS}=\{X\in\T\mid\T(X,S)=0\hbox{ for all }S\in\CS\}.\] Sometimes we write shortly 
$X\in\CS^\perp$ \iff $\T(\CS,X)=\{0\}$, and dual for the left hand perpendicular. 
We say that $\CS$ is 
 $\Sigma$\textsl{-stable} if it is closed under
suspensions and desuspensions, that is $\Sigma\CS\subseteq\CS$ 
and $\Sigma^{-1}\CS\subseteq\CS$. 

In definitions and remarks made in this paragraph we need $\T$ to be closed under coproducts for the direct notion 
and to be closed under products for the dual one. Recall that we say that the set of objects $\CS$ \textsl{generates} $\T$ if $\CS^\perp=\{0\}$. 
 Dually $\CS$ \textsl{cogenerates} $\T$ if $^\perp{\CS}=\{0\}$.
A \textsl{(co)localizing subcategory} of $\T$ is a triangulated subcategory which is closed under coproducts 
(respectively products). For $\CS\subseteq\T$ we denote $\Loc\CS$ and $\Coloc\CS$ the smallest 
(co)localizing subcategory containing $\CS$. It is obtained as the intersection of 
all (co)localizing subcategories which contain $\CS$. Note that if $\CS$ is $\Sigma$-stable, then 
$^\perp{\CS}$ (respectively ${\CS}^\perp$) 
is a (co)localizing subcategory. Moreover if $\CS$ is $\Sigma$-stable and $\T=\Loc\CS$ then 
$\CS$ generates $\T$, and dually $\T=\Coloc\CS$ implies $\CS$ cogenerates $\T$. 
Indeed, if we suppose $\CS^\perp\neq\{0\}$, then $^\perp{\left(\CS^\perp\right)}$ is 
localizing subcategory containing $\CS$ strictly smaller than $\T$. 

The first ingredient in the proof of the main theorem of this paper is contained in 
\cite{MD}. Here we recall it shortly.
If \[X_1\leftarrow X_2\leftarrow X_3\leftarrow \cdots \] is an inverse tower 
(indexed over $\N$) of objects in $\T$, where $\T$ is a triangulated category with products, then its 
homotopy limit is defined (up to a non--canonical isomorphism) by the triangle 
\[\holim X_n\la\prod_{n\in\N^*}X_n\stackrel{1-shift}\la\prod_{n\in\N^*}X_n\la\Sigma(\holim X_n), \]
(see \cite[dual of Definition 1.6.4]{N}).
Consider again a set of objects in $\T$ and denote it by $\CS$. We define $\Prod\CS$ to be the full
subcategory of $\T$ consisting of all direct factors of products of objects in $\CS$. 
Next we define inductively $\Prd_1(\CS)=\Prod\CS$ and $\Prd_n(\CS)$ is
the full subcategory of $\T$ which consists of all objects $Y$
lying in a triangle $X\to Y\to Z\to\Sigma X$ with $X\in\Prd_1(\CS)$
and $Y\in\Prd_n(\CS)$. Clearly the construction leads to an ascending chain
$\Prd_1(\CS)\subseteq\Prd_2(\CS)\subseteq\cdots$. Supposing $\CS$ to be 
$\Sigma$-stable, the same is true for
$\Prd_n(\CS)$, by \cite[Remark 07]{NR}.  The same
\cite[Remark 07]{NR} says, in addition, that if $X\to Y\to Z\to\Sigma X$ is a
triangle with $X\in\Prd_n(\CS)$ and $Z\in\Prd_m(\CS)$ then
$Y\in\Prd_{n+m}(\CS)$. An object $X\in\T$ will be called
$\CS$\textsl{-cofiltered} if it may be written as a homotopy limit 
$X\cong\holim X_n$ of an inverse tower, with $X_1\in\Prd_1(\CS)$, and
$X_{n+1}$ lying in a triangle $P_{n}\to X_{n+1}\to X_n\to\Sigma P_n,$
for some $P_n\in\Prd_1(\CS)$. Inductively we have
$X_n\in\Prd_n(\CS)$, for all $n\in\N^*$. 
The dual notion must surely be called \textsl{filtered}, and the terminology comes 
from the analogy with the notion of a filtered object in an abelian category (see \cite[Definition 3.1.1]{GT}). 
Using further the same analogy, we say that $\T$ (respectively, $\T\opp$) is \textsl{deconstructible} if 
$\T$ has coproducts (products) and there is a $\Sigma$-stable set $\CS\subseteq\T$, 
which is not a proper class 
such that every object $X\in\T$ is $\CS$-filtered (cofiltered). Note that we may define 
deconstructibility without assuming $\Sigma$-stability. Indeed 
if every $X\in\T$ is 
$\CS$-(co)filtered, then it is also 
$\overline{\CS}$-(co)filtered, where  $\overline{\CS}$ is the closure of $\CS$ 
under $\Sigma$ and $\Sigma^{-1}$. 

From now on, in  this section we fix $\T$ to be a triangulated category with products.  

\begin{prop}\label{decon}\cite[Theorem 8]{MD}
If $\T\opp$ is deconstructible, then $\T\opp$ satisfies Brown representability.
\end{prop}

In the next paragraph we shall work with ideals, so we recall their definition: An \textsl{ideal} 
in an additive category $\A$ is a collection of morphisms which is closed under addition and 
composition with arbitrary morphisms in $\A$. For $s\in\N^*$, the $s$\textsl{-th power} of an ideal $\I$ denoted
$\I^s$ is the ideal generated (that is the closure under addition) of the set 
\[\{f\mid\hbox{there are }f_1,\ldots,f_s\in\I\hbox{ such that }f=f_1\cdots f_s\}.\] 
If $\I$ and $\J$ are ideals, then to show that $\I^s\subseteq\J$ it is obviously enough to show that
the generators $f=f_1\cdots f_s$ lie in $\J$.
From now on $\CS\subseteq\T$ is a $\Sigma$-stable set. We call $\CS$-\textsl{(co)phantom} 
a map $\phi:X\to Y$ with the property $\T(\CS,\phi)=0$ (respectively $\T(\phi,\CS)=0$). 
(The notations $\T(\CS,\phi)=0$ and $\T(\phi,S)=0$ mean $\T(S,\phi)=0$, respectively 
$\T(\phi,S)=0$, for all $S\in\CS$.) Observe that 
$\phi:X\to Y$ is a phantom
\iff for every map $S\to X$ with $S\in\CS$, the composite map $S\to X\stackrel{\phi}\la Y$ 
vanishes, and dual for a cophantom. We denote
\[\Phi(\CS)=\{\phi\mid\phi\hbox{ is an }\CS\hbox{-phantom}\}\hbox{  and } 
\Psi(\CS)=\{\psi\mid\phi\hbox{ is an }\CS\hbox{-cophantom}\}.\] 
Clearly $\Phi(\CS)$ and $\Psi(\CS)$ are $\Sigma$-stable ideals in $\T$, that is they 
are also closed under $\Sigma$ and $\Sigma^{-1}$. Clearly, the ideals defined above depend on the ambient 
category $\T$. If we want to emphasize this dependence we shall write $\Phi_\T(\CS)$, respectively 
$\Psi_\T(\CS)$. 

For stating the following Lemma we need to recall what a preenvelope is: 
If $\A'\subseteq\A$ is a full subcategory of 
any category $\A$, then  
a map $Y\to Z$ in $\A$ with $Z\in\A'$ is called a $\A'$\textsl{-preenvelope of $Y$}, provided that 
every other map $Y\to Z'$ with $Z'\in\A'$ factors through $Y\to Z$. 

\begin{lem}\label{preen} If $\C$ is a set of objects in $\T$ then 
every $Y\in\T$ has a $\Prod\C$--preenvelope $Y\to Z$. Moreover if $\C$ is also $\Sigma$-stable,
then this preenveope fits in a triangle $X\overset{\psi}\to Y\to Z\to\Sigma X$, with $\psi\in\Psi(\C)$.
\end{lem}

\begin{proof}
 The argument is standard: Let $Z=\prod_{C\in\C,\alpha:Y\to C}C$ and $Y\to Z$ the unique map 
 making commutative the diagram:
 \[\diagram 
 Y\rrto\drto_{\alpha}&& Z\dlto^{p_{C,\alpha}}\\
 &C& 
 \enddiagram\] where $p_{C,\alpha}$ is the canonical projection for all $C\in\C$ and all 
 $\alpha:Y\to C$. For a $\Sigma$-stable set $\C$, we complete this map to a triangle \[X\overset{\psi}\to Y\to Z\to \Sigma X.\]
 It may be immediately seen that the condition to be a 
$\Prod\C$--preenvelope is equivalent to $\psi\in\Psi(\C)$.
\end{proof}

\begin{lem}\label{phi-triangle}
Assume that $\C\subseteq\T$ and $\G\subseteq\T$ are two $\Sigma$-stable sets, such that 
there is $s\in\N^*$ with the property $\Psi(\C)^s\subseteq\Phi(\G)$. Then every $Y\in\T$ 
fits in a triangle  $X\overset{\phi}\to Y\to Z\to\Sigma X$, with $Z\in\Prd_s(\C)$ and 
$\phi\in\Phi(\G)$.
\end{lem}

\begin{proof} We begin with an inductive construction. First denote $X_1=Y$, and if $X_k$ is already constructed, 
$k\in\N^*$, then use Lemma \ref{preen} to construct the triangle
\[X_{k+1}\overset{\psi_k}\to X_k\to P_k\to\Sigma X_{k+1},\]
where $X_k\to P_k$ is a $\Prod\C$--preenvelope of $X_k$ and $\psi_k\in\Psi(\C)$. Define also 
$\psi^1=1_Y:X_1\to Y$ and 
$\psi^{k+1}=\psi^k\psi_k:X_{k+1}\to Y$. Next complete them to triangles $X_k\overset{\psi^k}\to Y\to Z_k\to\Sigma X_k$, for all $k\in\N^*$. 
The octahedral 
axiom allows us to construct the commutative diagram whose rows and columns are triangles:
\[\diagram
             &\Sigma^{-1}Z_k\rdouble\dto&\Sigma^{-1}Z_k\dto &            \\
X_{k+1}\rto^{\psi_k}\ddouble&X_k\rto\dto^{\psi^k}        &P_k\rto\dto&\Sigma X_{k+1}\ddouble\\
X_{k+1}\rto_{\psi^{k+1}}    &Y\rto\dto         &Z_{k+1}\rto\dto  &\Sigma X_{k+1}        \\
             &Z_k\rdouble        &Z_k         &
\enddiagram.\]
 We have $Z_1=0$, $Z_2\cong P_1\in\Prod\C$ and the triangle in the second column of the above diagram 
 allows us to complete the induction step in order to show that $Z_{k+1}\in\Prd_k(\C)$. Clearly we also have $\psi^{k+1}\in\Psi(\C)^k$,
 thus the desired triangle is $X_{s+1}\overset{\psi^{s+1}}\la Y\la Z_{s+1}\la\Sigma X_{s+1}$. 
\end{proof}

\begin{lem}\label{same-im-fact} Assume that $\C\subseteq\T$ and $\G\subseteq\T$ are two $\Sigma$-stable sets, such that 
there is $s\in\N^*$ with the property $\Psi(\C)^s\subseteq\Phi(\G)$.
 Then every map $Y\to Z$ in $\T$ with $Z\in\Prd_n(\C)$ factors as $Y\to Z'\to Z$, where 
 $Z'\in\Prd_{n+s}(\C)$ and the induced maps 
 \[\T(G,Y)\to\T(G,Z)\hbox{ and }\T(G,Z')\to\T(G,Z)\]
 have the same image, for all $G\in\G$. 
\end{lem}

\begin{proof}
 Complete $Y\to Z$ to a triangle $Y\to Z\to Y'\to \Sigma Y$ and let 
 \[X\overset{\phi}\to Y'\to Z''\to\Sigma X,\] with 
 $\phi\in\Phi(\G)$ and $Z''\in\Prd_s(\C)$ the triangle whose existence is proved in Lemma 
 \ref{phi-triangle}. Complete the composed map $Z\to Y'\to Z''$ to a triangle 
 \[Z'\to Z\to Z''\to\Sigma Z'.\] It is clear that $Z'\in\Prd_{n+s}(\C)$. 
 We can construct the commutative diagram: 
 \[\diagram Y\rto\dto & Z\rto\ddouble & Y'\rto\dto &\Sigma Y\dto \\
            Z'\rto    & Z\rto         & Z''\rto    &\Sigma Z'
 \enddiagram\] by completing the middle square with $Y\to Z'$ in order to obtain 
 a morphism of triangles. 
 Applying the functor $\T(G,-)$ 
 with an arbitrary $G\in\G$ we get a commutative diagram with exact rows:
 \[\diagram \T(G,Y)\rto\dto & \T(G,Z)\rto\ddouble & \T(G,Y')\dto  \\
            \T(G,Z')\rto    & \T(G,Z)\rto         & \T(G,Z'')    
 \enddiagram.\] Since $\phi\in\Phi(\G)$ we deduce 
 $\T(G,Y')\to\T(G,Z'')$ is injective, so the kernels of the two right hand 
 parallel arrows are the same. But these kernels coincide to the images of the two left hand parallel arrows.   
\end{proof}

A diagram of triangulated categories and functors of the form $\CL\overset{I}\to\T\overset{Q}\to\U$ is called a
\textsl{localization sequence} if $I$ is fully faithful and has a right adjoint, $\Ker Q=\Img I$ and 
$Q$ has a right adjoint too. By formal non--sense (see 
\cite[Theorem 9.1.16]{N}) this right adjoint $Q_\rho$ of $Q$ has also to be fully faithful 
and makes $\U$ equivalent to the category $(\Img I)^\perp$. 

\begin{thm}\label{main-thm} Let $\G\subseteq\T$ be a $\Sigma$-stable set and denote $\U=\left(\G^\perp\right)^\perp$.
Suppose that there is a $\Sigma$-stable set $\C\subseteq\U$ and an integer $s\in\N^*$ such that 
$\Psi(\C)^s\subseteq\Phi(\G)$. Then $\U=\Coloc\C$, there is a localization sequence
$\G^\perp\to\T\to\U$ and $\U\opp$ satisfies Brown representability. 
\end{thm}

\begin{proof} Note first that, by its very construction,  
$\U$ is triangulated and closed under products in $\T$.
 Fix $Y\in\T$. Construct as in Lemma \ref{phi-triangle} a triangle $X_1\overset{\phi_1}\to Y\to Z_1\to\Sigma X_1$, with
$Z_1\in\Prd_s(\C)$ and $\phi_1\in\Phi(\G)$. We use Lemma \ref{same-im-fact} in order to inductively construct 
maps $Y\to Z_n$, with $Z_n\in\Prd_{sn}(\C)$, 
$n\in\N^*$, such that every $Y\to Z_n$ factors as $Y\to Z_{n+1}\to Z_n$, 
with the abelian group homomorphisms  
$\T(G,Y)\to\T(G,Z_n)$ and  $\T(G,Z_{n+1})\to\T(G,Z_n)$ having the same image, for all $G\in\G$. 
From now on the argument runs as in the proof of \cite[Theorem 4.7]{NKP}. We recall it here for the 
reader's convenience: Let $Z=\holim Z_n$.
 We have constructed the commutative diagram in $\T$:
 \[\diagram
  Y\rdouble\dto & Y\rdouble\dto & Y\rdouble\dto & \cdots \\
  Z_1 & Z_2\lto & Z_3\lto & \cdots\lto
\enddiagram\] inducing a map $Y\to Z$. 
Fix $G\in\G$. Applying the functor $\T(G,-)$ to above diagram in $\T$ we 
get a commutative diagram of abelian groups 
\[\diagram
  \T(G,Y)\rdouble\dto & \T(G,Y)\rdouble\dto & \T(G,Y)\rdouble\dto & \cdots \\
  \T(G,Z_1) & \T(G,Z_2)\lto & \T(G,Z_3)\lto & \cdots\lto
\enddiagram\]
with the first (hence all) vertical map(s) injective, and the images of both maps ending 
in each $\T(G,Z_n)$, $n\in\N^*$, coincide. This shows that the tower below is the direct sum 
of the above one and a tower with vanishing connecting maps, hence $\T(G,Y)\cong\lim\T(G,Z_n)$ 
canonically.
Moreover the inverse limit of the second row has to be exact, thus we obtain a short exact sequence: 
\[0\to\lim\T(G,Z_n)\la\prod_{n\in\N^*}\T(G,Z_n)\overset{1-shift}\la\prod_{n\in\N^*}\T(G,Z_n)\to0.\] 
  Comparing this sequence with the one obtained by applying $\T(G,-)$ to the 
 triangle \[Z\la\prod_{n\in\N^*}Z_n\overset{1-shift}\la\prod_{n\in\N^*}Z_n\la\Sigma Z\]
 we deduce $\lim\T(G,Z_n)\cong\T(G,Z)$. Therefore the map $Y\to Z$ constructed above induces 
 an isomorphism 
$\T(G,Y)\overset{\cong}\la\T(G,Z)$. Complete $Y\to Z$ to a triangle 
\[X\to Y\to Z\to\Sigma X.\]
Since $G$ was arbitrary, we deduce $X\in\G^\perp$ and obviously $Z\in\U$. Therefore 
the triangle above corroborated with \cite[Theorem 9.1.13]{N}
proves that there is a localization sequence $\G^\perp\to\T\to\U$. 
Finally supposing $Y\in\U$ this forces $X\in\U$, because $\U$ is triangulated. Since we have also 
$X\in\G^\perp$ we infer $X=0$, thus $Y\cong Z=\holim Z_n\in\Coloc\C$, hence 
$\U=\Coloc\C$ is $\CS$--cofiltered and all we need is to apply Proposition \ref{decon}.
\end{proof}

\begin{cor}\label{t-is-u} Assume that $\C\subseteq\T$ and $\G\subseteq\T$ are two $\Sigma$-stable sets, 
such that there is $s\in\N^*$ with the property $\Psi(\C)^s\subseteq\Phi(\G)$, and assume also 
that $\G$ generates $\T$. Then $\T=\Coloc\C$ and $\T\opp$ satisfies Brown representability.
\end{cor}

\begin{proof} The hypothesis $\G$ generates $\T$ means $\G^\perp=\{0\}$.
Thus one applies Theorem \ref{main-thm} with 
$\U=\left(\G^\perp\right)^\perp=\T$.
\end{proof}

\begin{cor}\label{gp-eq-u}
 Let $\G\subseteq\T$ be a $\Sigma$-stable set and denote $\U=\left(\G^\perp\right)^\perp$.
Suppose that there is a $\Sigma$-stable set $\C\subseteq\U$ and an integer $s\in\N^*$ such that 
$\Psi(\C)^s\subseteq\Phi(\G)$. Suppose in addition that $\T$ has coproducts and 
there is a localization sequence 
$\Loc\G\to\T\to\G^\perp$. Then $\Loc\G$ is equivalent to $\U$ and, consequently, 
${\Loc\G}\opp$ 
satisfies Brown representability. In particular, a localization sequence as above exists, 
provided that objects in $\G$ are 
$\alpha$-compact, for a regular cardinal $\alpha$.
\end{cor}

\begin{proof} First apply Theorem \ref{main-thm} in order to obtain a localization 
sequence $\G^\perp\to\T\to\U$. Together with the localization sequence whose existence 
is supposed in the hypothesis, this shows that both 
 categories $\Loc\G$ and $\U$ are equivalent to the Verdier quotient $\T/\G^\perp$, hence they are 
 equivalent to each other. Finally provided that objects in $\G$ are $\alpha$-compact, we know 
 by \cite[Theorem 8.4.2]{N} that $\Loc\G$ satisfies 
 Brown representability. 
 Consequently the inclusion functor $\Loc\G\to\T$ which preserves coproducts must have 
 a right adjoint and a localization sequence 
 $\Loc\G\to\T\to\G^\perp$ exists.
\end{proof}

In the end of this section let observe that the general version of Brown representability for 
covariant functors proved in \cite{K} is a consequence of our criterion. 
In order to do that, let $\T$ be a triangulated category with products and coproducts. Recall 
from \cite[Definition 2]{K} that a \textsl{set of symmetric generators} for $\T$ is a set $\G\subseteq\T$
which generates $\T$ such that  and there is another set $\C\subseteq\T$ with the property 
that for every map $X\to Y$ in $\T$ the induced map $\T(G,X)\to\T(G,Y)$ is surjective 
for all $G\in\G$  \iff the induced map $\T(Y,C)\to\T(X,C)$ is injective for all $C\in\C$. 
Completing the map $X\to Y$ to a triangle it is easy to see that the last condition is equivalent 
to the fact $\Phi(\G)=\Psi(\C)$. Remark also that without losing the generality, we may suppose 
the sets $\G$ and $\C$ to be $\Sigma$-closed. Applying Corollary \ref{t-is-u} we obtain:

\begin{cor}\label{krauseB}\cite[Theorem B]{K}
 If $\T$ has products, coproducts and a set of symmetric generators, then $\T\opp$ satisfies Brown representability.
\end{cor}

Remark also 
that hypotheses in \cite[Theorem B]{K} are general enough to include the case of compactly 
generated categories.

\section{Applications: Homotopy categories of projective representations of quivers}\label{applications}

Let $\A$ be an additive category. \textsl{Complexes (cohomologically graded)} over 
$\A$ are sequences of the form
\[X=\cdots\to X^{n-1}\stackrel{d^{n-1}}\la X^n\stackrel{d^n}\la X^{n+1}\to\cdots\]
with $X^n\in\A$, $n\in\Z$, and $d^nd^{n-1}=0$; these morphisms are called \textsl{differentials}.
Morphisms of complexes are collections of morphisms in $\A$ commuting with differentials. In this way 
complexes over $\A$ form a category denoted $\Cpx\A$. Limits and colimits in the category $\Cpx\A$ 
are computed component--wise, provided that the respective constructions may be performed in $\A$.
In particular $\Cpx\A$ is abelian (Grothendieck) if $\A$ is so.

Two maps of complexes $(f^n)_{n\in\Z},(g^n)_{n\in\Z}:X\to Y$ are 
homotopically equivalent if there are $s^n:X^n\to Y^{n-1}$, for all
$n\in\Z$, such that $f^n-g^n=d_Y^{n-1}s^n+s^{n+1}d_X^n$.
The homotopy category $\Htp\A$ has as objects 
all complexes and as morphisms equivalence classes of morphisms of complexes up to homotopy. 
It is well--known 
that $\Htp\A$ is a triangulated category with (co)products, provided that the same property 
is valid for $\A$.

Recall that a \textsl{ring with several objects} is a small preadditive category $R$, and  an $R$-module 
is a functor $R\opp\to\Ab$. (Our modules are right modules 
by default.) Clearly if the category $R$ has exactly one object, then it is nothing else than an
ordinary ring with unit, and modules are abelian groups endowed with a multiplication with scalars from $R$. 
In the sequel, the category $\A$ will be often an additive exact (that is closed under 
extensions) subcategory of the category 
$\ModR$ of modules over a ring with several objects $R$. For example, $\A$ may be $\FlatR$ of $\ProjR$ the full subcategories of all flat, 
respectively projective modules. Another 
source of examples  is the subcategory of projective complexes over a module category $R$, 
that is $\Proj{\Cpx R}$. Note then that if $\A\subseteq\ModR$ or $\A\subseteq\Cpx\ModR$ 
an additive exact category as above, then 
$\Htp\A$ is triangulated subcategory of $\Htp\ModR$ respectively $\Htp{\Cpx\ModR}$.

The general argument in the Section above is obtained by formalization of \cite{MKP}, 
where the main result is that the homotopy category of projective modules satisfies 
Brown representability for covariant functors. Consequently it is not  
surprisingly at all that an application of Section \ref{general-results} is:

\begin{thm}\label{kproj-mod}\cite[Theorem 1]{MKP}
The homotopy category of projective modules over a ring with several objects
satisfies Brown representability for covariant functors.
\end{thm}

\begin{proof} 
 We only have to verify that the assumptions made in Corollary 
 \ref{gp-eq-u} are fulfilled. Fix a ring with several objects $R$ and denote $\T=\Htp\FlatR$.
 Note that, even if the 
 results in \cite{NKF} and \cite{NKP} are stated for rings with one, they don't make 
 use of the existence of the unit, and the same arguments may be used for rings with several 
 objects. 
 
 Let $\G$ be the closure under $\Sigma$ of the generating set for $\Htp\ProjR$ defined in 
 \cite[Construction 4.3]{NKF} of \cite[Reminder 1.5]{NKP}. 
 According to \cite[Theorem 5.9]{NKF},
 $\G$ is a a set of $\aleph_1$-compact generators for $\Htp\ProjR$. Denote 
 $\U=\left(\G^\perp\right)^\perp$. Further let $\C$ be 
 the closure under $\Sigma$ and $\Sigma^{-1}$ of the set with elements $J(\Hom_\Z(I,\Q/\Z))$ where $I$ runs 
 over all \textsl{test-complexes} of left $R$-modules in the sense of 
 \cite[Definition 1.1]{NKP}, and $J:\Htp\ModR\to\Htp\FlatR=\T$ is the right adjoint of the 
 inclusion functor $\Htp\FlatR\to\Htp\ModR$, which exists by \cite[Proposition 8.1]{NKF}. 
 Then \cite[Lemma 2.6]{NKP} implies $\C\subseteq\U$. 
 According to \cite[Lemma 2.8]{NKP}, $\Psi(\C)$ contains exactly the so called 
 \textsl{tensor-phantom} maps in 
 the sense of \cite[Definition 1.3]{NKP}. Finally $\Psi(\C)^2\subseteq\Phi(\G)$, as 
 \cite[Lemma 1.9]{NKP} states. Therefore Corollary \ref{gp-eq-u} applies, 
 hence $\Htp\ProjR=\Loc\G$ satisfies Brown representability 
 for covariant functors.
\end{proof}

\begin{cor}\label{pure-proj}
 If $R$ is a ring, then homotopy category of pure--projective modules satisfies 
 Brown representability for covariant functors. 
\end{cor}

\begin{proof} Let $R$ be a ring and denote by $\Pproj R$ the category of pure 
 projective $R$-modules. 
 Denote by $A=\md R$ the full subcategory of $\ModR$ which consists of all finitely presented 
 modules, and see it as a ring with several objects. It is well known that the functor 
 \[\ModR\to\ModA,\ X\mapsto\Hom_R(-,X)|_A\] is an embedding and restrict to an 
 equivalence $\Pproj R\overset{\sim}\to\Proj{A}$. Now apply Theorem \ref{kproj-mod}.
\end{proof}

Recall that a ring is called \textsl{pure--semisimple} if every $R$-module is pure projective. 
Therefore from Corollary \ref{pure-proj} we may derive a new proof for an 
already known result (see \cite[Theorem 10 and Remark 11]{MD}):

\begin{cor}\label{pure-semisimple}
 If the ring $R$ is pure--semisimple then $\Htp\ModR$ satisfies Brown 
 representability for covariant functors. 
\end{cor}

The following proposition gives a method for constructing deconstructible triangulated categories, 
starting with one which satisfies the same property:  

\begin{prop}\label{ff-right-adj}
Let $U:\CL\to\T$ be a fully faithful functor which has a right adjoint $F:\T\to\CL$. If $\T\opp$ is deconstructible, then $\CL\opp$ is so and,
consequently, $\CL\opp$ satisfies Brown representability.
\end{prop}

\begin{proof}
By hypothesis there is a set $\C$ such that $\T$ is $\C$--cofiltered. We shall show that $\CL$ is $F(\C)$-filtered, and we are done. 
Let $X\in\CL$. Then for $U(X)\in\T$ there is an inverse tower \[0=X_0\leftarrow X_1\leftarrow\cdots\] such that $U(X)\cong\holim_{n\in\N}X_n$ and
in the triangle $X_{n+1}\to X_n\to P_n\to \Sigma X_{n+1}$ we have $P_n\in\Prod\C$. Applying the product preserving triangulated functor $F$ we obtain 
$FU(X)\cong\holim_{n\in\N}F(X_n)$ and in the triangle $F(X_{n+1})\to F(X_n)\to F(P_n)\to \Sigma F(X_{n+1})$ we have $F(P_n)\in\Prod{F(\C)}$. Finally
it remains only to note that $X\cong FU(X)$ since $U$ is supposed fully faithful. 
\end{proof}

Recall that a \textsl{quiver} is a quadruple 
 $Q=(Q_0, Q_1, s, t)$ where $Q_0$ and $Q_1$ are disjoint sets whose elements are called \textsl{vertices}, 
 respectively \textsl{arrows} of $Q$, and $s,t:Q_1\to Q_0$ are two maps. If for $m\in Q_1$ 
 we have $i=s(m)$ and $j=t(m)$ 
 then we call the vertices $i$ and $j$ the source, respectively the target of the arrow $m$. 
 We write $m:i\to j$ to indicate this fact. Two arrows $m,m'\in Q_1$ are composable if 
 $s(m')=t(m)$. If this is the case we denote by $m'm$ the composition and we set $s(m'm)=s(m)$
 and $t(m'm)=t(m')$. We have just obtained a path of length $2$. Generalizing this, 
 a \textsl{path} in the quiver $Q$ is a finite sequence of composable maps; the number of the maps occuring 
 in a path is called the \textsl{length} of the path. Vertices are seen as paths of length 0, 
 or trivial paths.
 A relation in a quiver is obtained as following: Consider $\{(\gamma_i,\delta_i)\mid i\in I\}$ 
 an arbitrary set of pair of paths such that $s(\gamma_i)=s(\delta_i)$ and $t(\gamma_i)=t(\delta_i)$ 
 for all $i\in I$. We put $\gamma_i\sim\delta_i$, and whenever $\sigma$ and $\tau$ 
 are paths such that the compositions make sense we have $\sigma\gamma_i\sim\sigma\delta_i$
 and $\gamma_i\tau\sim\delta_i\tau$. Moreover for any path $\gamma$ we set $\gamma\sim\gamma$. 
It is easy to see that $\sim$ will be an equivalence 
 relation on the set of all paths in $Q$. It is clear that every quiver may be seen as a quiver 
 with relations, since equality is the poorest equivalence relation. Henceforth by a quiver we will always mean a quiver with relations.  
 A representation $X$ of the quiver $Q$ in $\ModR$ is an assignment to each vertex $i\in Q_0$ an
 $R$-module $X(i)$ and to each arrow $m:i\to j$ in $Q_1$ an $R$-linear map $X(m):X(i)\to X(j)$, 
 such that equivalent paths lead to equal linear maps. Morphisms of representations $f:X\to Y$ are collections of $R$-linear maps $f=(f_i:X(i)\to Y(i))_{i\in Q_0}$, 
with the property $f_jX(m)=Y(m)f_i$ for any arrow $m:i\to j$ in $Q_1$. 
 We obtain a category, namely the category of representations of $Q$ in $\ModR$ denoted $\Mod{R,Q}$.
 Further, let $A$ be the free $R$-module with 
 the basis $B$ the set of all paths in $Q$ modulo the equivalence relation above, 
 that is \[A=R^{(B)}=\bigoplus_{b\in B}bR,\] where 
 $bR=\{br\mid r\in R\}$ is a copy of $R$ as right $R$-module. For two elements $b=[\gamma]$ and 
 $b'=[\gamma']$ in $B$ we define the product $b'b=[\gamma'\gamma]$ if the paths $\gamma$ and 
 $\gamma'$ 
 are composable, and $b'b=0$ otherwise. Declaring that elements in $R$ 
 commute with all elements of the base, the product extends by distributivity 
 to all elements in $A$, making $A$ into an $R$-algebra, the so called \textsl{path algebra} of $Q$ 
 over $R$. The trivial paths lead to a family of orthogonal idempotents $e=e_i\in A$, 
 with $i\in Q_0$. If $Q_0$ is finite, then $\sum_{i\in Q_0}e_i$ is unit in $A$, otherwise we may 
 add an extra element $1$ to the basis $B$ which acts as unit, that is $1b=b1=b$ for any path $b$.
 A slightly different (but equivalent) approach of the matter concerning quivers, 
 may be found in \cite[Chapter II,\S1]{M} (quivers are called there diagram schemes). 
 The categories $\Mod{R,Q}$ and $\ModA$ are linked by two functors, namely $U:\Mod{R,Q}\to\ModA$, 
 given by $U(X)=\bigoplus_{i\in Q_0}X(i)$, and $F:\ModA\to\Mod{R,Q}$, $F(M)(i)=Me_i$. 

 \begin{lem}\label{u-f} 
 With the above notations the following statements hold:

\begin{itemize}
\item[{\rm a)}] The functor $F$ is the right adjoint of $U$. 
\item[{\rm b)}] $U$ is fully faithful.
\item[{\rm c)}] Both $F$ and $U$ preserve projective objects.
\end{itemize}
\end{lem}

\begin{proof} a) Let $X\in\Mod{R,Q}$ and $M\in\ModA$. 
 If $f:\bigoplus_{i\in Q_0}X(i)\to M$ in an $A$-linear map, then for every $x\in X(i)$ we have 
 $f(x)=f(xe_i)=f(x)e_i\in Me_i$, showing that $f=(f_i:X(i)\to Me_i)_{i\in Q_0}$. It is not 
 hard to see that the $R$-linear maps $f_i$ have to commute with the maps induced by every 
 $i\to j$ in $Q_1$, so $f$ is a map of representations $X\to F(M)$. Conversely if $f_i:X(i)\to Me_i)_{i\in I}$ 
is a map of representations, then the family of maps $(X(i)\to Me_i\to M)_{i\in I}$ induce a unique $R$-linear map 
$f:\bigoplus_{i\in I}X_i\to M$. Moreover if $a\in A$ and $x\in\bigoplus_{i\in I}X_i$ then both are written as 
finite sums $a=\sum_{b\in B}ba_b$ and $x=\sum_{i\in I}x_i$ with $a_b\in R$ and $x_i\in X(i)$ (almost all zero).
By distributivity $f(xa)=f(x)a$, hence $f$ is $A$-linear.  
This proves the adjunction between $F$ and $U$.  

b) Let $X$ be a representation of $Q$. We have 
\[F(U(X))=F\left(\bigoplus_{i\in I}X(i)\right)=(X(i))_{i\in I}=X \]  therefore the functor $U$ is fully faithful. 

c) First observe that for $M\in\ModA$,  we have 
$F(M)(i)=Me_i$, $i\in I$, and $Me_i$ is a direct summand of $M$, hence $F$ is an 
exact functor. This implies that its left adjoint $U$ preserves projective objects. Moreover $F$ preserves 
coproducts. So for showing that $F$ preserves projective objects it is enough to show that $F(A)$ is projective in $\Mod{R,Q}$. In order to prove this we will determine better the projective objects in 
$\Mod{R,Q}$. View $Q$ as a small category with object set $Q_0$ and maps equivalence classes of paths in $Q$.
By \cite[Chapter II, \S 12]{M}, $\Mod{R,Q}$ is equivalent to the category of functors from this small category to $\ModR$, consequently 
according to \cite[Chapter VI, Theorem 4.3]{M}, projectives in $\Mod{R,Q}$ are exactly 
\[\Proj{(R,Q)}=\Add{\{S_i(P)\mid i\in Q_0\hbox{ and }P\in\ProjR\}},\]
where $S_i:\ModR\to\Mod{R,Q}$ are functors defined by \[S_i(V)=\bigoplus_{i/Q_0}V\hbox{ for all }V\in\ModR,\] 
$i/Q_0=\{([\gamma],j)\mid j\in Q_0\hbox{ and }\gamma:i\to j\hbox{ is a path}\}$, and by $\mathrm{Add}$ we 
understand the closure under direct sums and direct summands (that is, the dual of $\mathrm{Prod}$). Now, since $F(A)(i)=Ae_i$, for all $i\in Q_0$, we deduce $F(A)\in\Proj{(R,Q)}$.
\end{proof}

\begin{thm}\label{kproj-ri}
 Let $Q$ be a quiver and denote by $\Proj{(R,Q)}$ the category of projective 
 objects in the category $\Mod{R,Q}$. Then $\Htp{\Proj{(R,Q)}}$ 
 satisfies Brown representability for covariant functors. 
\end{thm}

\begin{proof} 
Consider the path algebra $A$ of the quiver $Q$ and the pair of 
adjoint functors $U:\Mod{R,\I}\leftrightarrows\ModA:F$ defined above.  By Lemma \ref{u-f} we obtained a pair of adjoint functors 
between $\Proj{(R,Q)}$ and $\Proj{A}$, which extends to a pair of triangulated adjoint
functors (denoted with the same symbols) \[U:\Htp{\Proj{(R,Q)}}\leftrightarrows\Htp{\Proj{A}}:F\hbox{ .}\] 
In addition we know that the initial $U$ is fully faithful, so the same is true for the extended functor.  By Theorem \ref{kproj-mod},
$\Htp{\Proj{A}}$ is deconstructible,
hence Proposition \ref{ff-right-adj} applies. 
\end{proof}

\begin{expl}\label{kproj-cpx}
If in Theorem \ref{kproj-ri} we put $Q$ to be the following quiver
\[\cdots\to i-1\overset{\partial^{i-1}}\to i\overset{\partial^i}\to i+1\to\cdots,\ (i\in\Z),\] with 
relations $\partial^i\partial^{i-1}=0$, then a representation of this quiver is, obviously, 
a complex over $R$. Therefore $\Mod{R,Q}=\Cpx\ModR$, and we obtain a proof for the fact that 
the homotopy category of projective complexes of $R$-modules satisfies Brown representability 
for covariant functors.   
\end{expl}

\section*{Acknowledgements} The author thanks to Simion Breaz for asking him if Corollary \ref{pure-semisimple} 
may be derived from Theorem \ref{kproj-mod}.
He is also indebted to an anonymous referee for many comments leading to a substantial 
improving of the presentation.

\end{document}